\documentclass[reqno]{amsart}
\usepackage{hyperref}
\usepackage{amssymb,amsmath}
\usepackage{amsthm}

\newtheorem{theorem}{Theorem}[section]

\newtheorem{lemma}[theorem]{Lemma}

\theoremstyle{definition}

\DeclareMathOperator{\dif}{d}

\newcommand{\pp}[2]{\frac{\partial{#1}}{\partial{#2}}}

\numberwithin{equation}{section}

\title{A conclusive theorem on Finsler metrics of sectional flag curvature}
\author{Libing Huang \and Zhongmin Shen}
\address[Libing Huang]{School of Mathematical Sciences and LPMC,  Nankai University, 
Tianjin 300071, P.~R.~China}
\email{huanglb@nankai.edu.cn}
\address[Zhongmin Shen]{Department of Mathematical Sciences,
Indiana University-Purdue University,
Indianapolis, IN 46202-3216, USA}
\email{zshen@math.iupui.edu}
\date{\today}

\begin{document}

\begin{abstract}
If the flag curvature of a Finsler manifold reduces to sectional curvature,  then locally either the 
Finsler metric is Riemannian, or the flag curvature is isotropic.
\end{abstract}

\maketitle

\section{Introduction}

Flag curvature is the most important quantity in Finsler geometry.  It is the analogue of sectional curvature
in Riemannian geometry.  

The flag curvature $K$ is a function defined on flags.  At a base point $x$ on the manifold $M$,  a 
flag $(y, P)$ consists of a flagpole $y\in T_xM$ and a two dimensional section $P$ containing $y$. 
If $P$ is spanned by $\{y,v\}$, then the flag curvature $K(y,P)$ can be also written as $K(y,v)$.
For precise definition,  one may consult section \ref{sec:pre} or standard textbooks such 
as \cite{BCS} or \cite{ChernShen}. 

Of particular interest is the study of Finsler metrics of constant flag curvature.  It is a remarkable progress to
classify Randers metrics of constant flag curvature in Bao-Robles-Shen \cite{BRS}.  At the same time,  many new examples
of constantly curved non-Randers metrics are found, see for example \cite{Bryant}, \cite{Shen}, etc.  However,
the knowledge of the local structure of constantly curved Finsler metrics is still rare.

If the flag curvatures of any flags based at $x$ are equal, then the Finsler metric is said to have
\emph{isotropic flag curvature} at $x$.  In dimensions at least three,  there is the Finsler version of Schur's lemma 
asserting that everywhere isotropic flag curvature implies constant flag curvature.

If the flag curvature only depends on flagpoles, then the Finsler metric is said to be of scalar flag curvature, since
in this case the flag curvature is a scalar function on the slit tangent bundle.  If the flag curvature only depends on
sections, but does not depend on flagpoles,  then the Finsler metric is said be of sectional flag curvature.
Clearly,  Riemannian metrics are of sectional flag curvature;  isotropically or constantly curved Finsler metrics 
are also of sectional flag curvature.  These examples are trivial.  Thus it is natural to ask: 

\emph{Are there any nontrivial Finsler metrics with sectional flag curvature?}

Bin Chen and Lili Zhao first studied this problem for Randers metrics.  It is proved in \cite{ChenZhao2010} that
a non-Riemannian Randers metric of sectional flag curvature must be constantly curved if the dimension is greater than two.  
In their proof,   Bao-Robles's curvature formula \cite{BR2004} is crucial,  and the result follows by complicated computation.
Later,  they obtained some Numata type and Akbar-Zadeh type theorems in \cite{ChenZhaoDGA}. 

There are two goals in this short note.  The first goal is to provide a simple proof of Chen-Zhao's result \cite{ChenZhao2010}
on Randers metrics,  which turns out to be essential for the general case.   The second goal is to prove the following 
conclusive result,  which gives a negative answer to the above problem.

\begin{theorem}\label{thm:main}
Let $(M,F)$ be a Finsler manifold of sectional flag curvature.  Then at any point $x$ in $M$, either
the Minkowski norm $F|_x$ is Euclidean,  or the flag curvature is isotropic at $x$.
\end{theorem}

Notice that when $\dim M = 2$,  the flag curvature $K$ is automatically a scalar function on the tangent bundle.  
If it does not depend on the flagpole $y$,  then it is a function on $M$;  namely, the flag curvature is isotropic and the theorem is 
trivially true.  In the following we will only consider the case $\dim M \geq 3$.

The paper is organized as follows.  In section \ref{sec:pre},  we review some basic material of Finsler geometry and list
the tools that is needed in context.  In section \ref{sec:lem}, we establish an equation to characterize Finsler metrics
of sectional flag curvature, and explore several consequences of this equation.  Finally we give the proofs of Chen-Zhao's result
and theorem \ref{thm:main} in section \ref{sec:prf}.

\section{Preliminaries}
\label{sec:pre}

Let $M$ be an $n$ dimensional smooth manifold with tangent bundle $TM$.   A \emph{Finsler metric} $F=F(x,y)$ on a 
manifold $M$ is a positive smooth function on $TM\backslash\{0\}$ such that (1) $F(x,\lambda y)=\lambda F(x,y)$ for
any $\lambda>0$; (2) the quadratic form
\[
	g_y = g_{ij}(x,y) dx^i\otimes dx^j,\quad g_{ij} = \frac{1}{2}[F^2]_{y^iy^j}
\]
is positive definite.  Sometimes we write $F(y)$ instead of $F(x,y)$ because the dependence on $x$ is implicitly
declared since the tangent vector $y$ must have a base point $x$.  The restriction of a Finsler metric $F$ to each tangent space
is a Minkowski norm.   In particular, if all the Minkowski norms are Euclidean,  then we recover the definition of a Riemannian 
metric.

Let $(g^{ij})$ be the inverse matrix of $(g_{ij})$.  The \emph{Cartan tensor} 
$\mathbf{C}_y = C_{ijk}\dif x^i\otimes\dif x^j\otimes\dif x^k$ and \emph{mean Cartan tensor} $\mathbf{I}_y = I_i dx^i$ 
are defined respectively by
\[
	C_{ijk} = \frac{1}{4}[F^2]_{y^iy^jy^k} = \frac{1}{2}[g_{ij}]_{y^k},\quad 
	I_i = g^{jk} C_{ijk}.
\]
Clearly, $F$ is a Riemannian metric if and only if $\mathbf{C}_y=0$ holds for all $y$.
A theorem of Deick asserts the condition $\mathbf{I}_y=0$ also characterizes Riemannian metrics (see \cite{BCS}).

The \emph{angular metric} $h_y = h_{ij} dx^i\otimes dx^j$ is defined by
\begin{equation}
	h_{ij} = g_{ij} - F_{y^i}F_{y^j}.
\end{equation}

M. Matsumoto introduced the notion of $\mathbf{C}$-reducible Finsler metrics, which are characterized by the vanishing
of the \emph{Matsumoto torsion} $\mathbf{M}_y=M_{ijk}dx^i\otimes dx^j\otimes dx^k$, where
\begin{equation}
	M_{ijk} = C_{ijk} - \frac{1}{n+1}(h_{ij}I_k + h_{jk}I_i + h_{ki}I_j).
\end{equation}
It is proved in \cite{Matsumoto} that in dimensions greater than two, $F$ is $\mathbf{C}$-reducible, if and only if it is of Randers type, 
namely, $F$ can be written as $\alpha+\beta$, where $\alpha$ is a Riemannian metric and $\beta$ is a one form 
with $|\beta|_{\alpha}<1$.

The spray $G=y^i\pp{}{x^i} - 2G^i\pp{}{y^i}$ is a vector field on $TM\backslash\{0\}$,
where the coefficients $G^i=G^i(x,y)$ are defined by
\[
	G^i = \frac{1}{4}g^{il}\Big\{[F^2]_{x^ky^l}y^k - [F^2]_{x^l}\Big\}.
\]

The importance of spray in Finsler geometry can be illustrated by the following two facts.
The first fact is that the \emph{Riemann curvature tensor} $\mathbf{R}_y = R^i{}_k\pp{}{x^i}\otimes dx^k$ can be
written as a combination of derivatives of $G^i$ as follows
\begin{equation}
	R^i{}_k = 2[G^i]_{x^k} - y^j[G^i]_{x^jy^k} + 2G^j[G^i]_{y^jy^k} - [G^i]_{y^j}[G^j]_{y^k}.
\end{equation}
Notice that $\mathbf{R}_y$ is a $(1,1)$ tensor.  Let $R_{jk} = g_{ij}R^i{}_k$.  Then we also denote the $(0,2)$ 
tensor $R_{jk}dx^j\otimes dx^k$ by $\mathbf{R}_y$, and also call it \emph{Riemann curvature tensor}.  One can
distinguish these two variants by context.  An important property of the $(0,2)$ tensor 
$\mathbf{R}_y$ is that it is symmetric, $R_{jk}=R_{kj}$.

The second fact is that the spray can be used to define a \emph{dynamical derivative}.  For example,
the dynamical derivative of the Cartan tensor is $\dot{\mathbf{C}}_y = C_{ijk|0} dx^i \otimes dx^j \otimes dx^k$, where
\[
	C_{ijk|0} = G(C_{ijk}) - [G^m]_{y^i}C_{mjk} - [G^m]_{y^j}C_{imk} - [G^m]_{y^k}C_{ijm}.
\]
Traditionally, the tensor $\dot{\mathbf{C}}_y$ is also called \emph{Landsberg curvature}.
As other examples,  one can verify that $\dot{g}_y$ and $\dot{h}_y$,  the dynamical derivatives of 
the fundamental tensor $g_y$ and the angular metric $h_y$,  are zero.

Finally,  if a section $P$ is spanned by the flagpole $y$ and another vector $v$, then the flag curvature $K(y,P)$ is given by
\begin{equation}
	K(y,P) = K(y,v) = \frac{\mathbf{R}_y(v,v)}{F(y)^2 h_y(v,v)}.
\end{equation}

The above facts are standard in textbooks (cf. \cite{ChernShen}).  Now we provide some notations that will be used in next section.
The vertical derivative of $\mathbf{R}_y$ will be denoted by $\hat{\mathbf{R}}_y$.  It is defined by 
$\hat{\mathbf{R}}_y = R{}_{ij\cdot k}dx^i\otimes dx^j\otimes dx^k$,  where
\[
	R{}_{ij\cdot k} = [R_{ij}]_{y^k}.
\]
It is easy to see that
\begin{equation}
	\hat{\mathbf{R}}_y(w,u,v) = \frac{d}{dt}\mathbf{R}_{y+tv}(w,u)\Big|_{t=0}.
\end{equation}

To build a relation between $\hat{\mathbf{R}}_y$ and other quantities,  we recall an important identity in \cite{ChernShen}.  
Lemma 2.4.1 in \cite{ChernShen} states that
\begin{equation}
\begin{aligned}
	C_{ijk|0|0} + C_{ijm}R^m_k = &-\frac{1}{3}g_{im}R^m{}_{k\cdot j} - \frac{1}{3}g_{jm} R^m{}_{k\cdot i}\\
		& -\frac{1}{6}g_{im}R^m{}_{j\cdot k} - \frac{1}{6}g_{jm} R^m{}_{i\cdot k},
\end{aligned}
\end{equation}
where $R^i{}_{j\cdot k} = [R^i{}_j]_{y^k}$.  Notice that 
$R_{ik\cdot j} = (g_{im}R^m{}_k)_{\cdot j} = 2C_{ijm}R^m{}_k + g_{im}R^m{}_{k\cdot j}$,  the above 
equation can be rewritten as
\begin{equation}\label{eq:}
\begin{aligned}
	& R_{ik\cdot j} + R_{jk\cdot i} + R_{ij\cdot k}\\  
	=& C_{ijm}R^m_k  + C_{jkm} R^m{}_i + C_{kim} R^m{}_j - 3C_{ijk|0|0}.
\end{aligned}
\end{equation}
We encode this equation into the following index-free form.

\begin{lemma}\label{lem:extR}
$\hat{\mathbf{R}}_y(v,v,v) = \mathbf{C}_y(v,v,\mathbf{R}_y(v)) - \ddot{\mathbf{C}}_y(v,v,v)$.
\end{lemma}

At the end of this section, we prove

\begin{lemma}\label{lem:exth}
$\frac{d}{dt}\Big(F^2(y+tv)h_{y+tv}(v,v)\Big)_{t=0} = 2F^2(y)\mathbf{C}_y(v,v,v)$. 
\end{lemma}
\begin{proof}
The left hand side tensor has coefficents $(F^2h_{ij})_{\cdot k}$.  Direct computation yields
\begin{align*}
	(F^2 h_{ij})_{\cdot k} &= \big(F^2(g_{ij} - F_{y^i} F_{y^j})\big)_{\cdot k} \\
	&= 2F^2 C_{ijk} + 2FF_{y^k} g_{ij}  - g_{ik} FF_{y^j} - g_{jk} FF_{y^i}.
\end{align*}
Contraction with $v^iv^jv^k$ then gives the desired relation.
\end{proof}

\section{Some lemmas}
\label{sec:lem}

Finsler metrics of sectional flag curvature can be characterized by the following lemma.  An equivalent version
of this lemma can be found in Chen-Zhao \cite{ChenZhaoDGA}.

\begin{lemma}\label{lem:eq}
A Finsler manifold $(M,F)$ is of sectional flag curvature, if and only if
\begin{equation}\label{eq:Rh=CR}
	\big(\mathbf{C}_y(v,v,\mathbf{R}_y(v)) - \ddot{\mathbf{C}}_y(v,v,v)\big)\cdot h_y(v,v) = 2\mathbf{C}_y(v,v,v)\mathbf{R}_y(v,v)
\end{equation}
holds for any $y$, $v$.
\end{lemma}
\begin{proof}
It is easy to see that, $F$ is of sectional flag curvature, if and only if
\[
	\left.\frac{d}{dt}K(y+tv,v)\right|_{t=0} = 0
\] 
holds for any linearly independent vectors $y$ and $v$ (see Lemma 2.2 in \cite{ChenZhaoDGA}).  
Since $K(y,v) = \frac{\mathbf{R}_y(v,v)}{F^2(y) h_y(v,v)}$,  the above equation can be expanded to get
\[
	\left.\frac{d}{dt}\mathbf{R}_{y+tv}(v,v)\right|_{t=0}\cdot F^2h_y(v,v) = \frac{d}{dt}\Big(F^2(y+tv)h_{y+tv}(v,v)\Big)_{t=0} \mathbf{R}_y(v,v).
\]
Substituting the results of lemma \ref{lem:extR} and lemma \ref{lem:exth} into the above relation, 
we proved (\ref{eq:Rh=CR}).
\end{proof}

Now, fix a nonzero tangent vector $y$ in $T_xM$.  
If there is a constant $\kappa=\kappa(y)$ such that $\mathbf{R}_y(v,v) = \kappa(y) F^2(y) h_y(v,v)$ holds for any $v$ in $T_xM$, 
then we call $y$ a \emph{polar direction}.

\begin{lemma}\label{lem:key}
Let $(M,F)$ be a Finsler manifold of sectional flag curvature with $\dim M\geq 3$.  For each fixed $x$ in $M$,
every nonzero vector $y$ in $T_xM$ must satisfy at least one of the following two conditions: (1) $y$ is a polar direction 
or (2) the Matsumoto torsion $\mathbf{M}_y$ vanishes.
\end{lemma}
\begin{proof}
Both sides of equation (\ref{eq:Rh=CR}) are polynomials
in the variable $v$.  Since $h_y$ is semi-positive of rank $\dim M -1$,  we find that $h_y(v,v)$ is irreducible when 
$\dim M\geq 3$.  Thus, (\ref{eq:Rh=CR}) implies that either $h_y(v,v)\,|\,\mathbf{R}_y(v,v)$ or $h_y(v,v)\,|\,\mathbf{C}_y(v,v,v)$.
The former is equivalent to saying that $y$ is a polar direction.  The latter is equivalent to saying that the 
Cartan tensor $\mathbf{C}_y$ is reducible.  It is well known that in this case 
\[
	C_{ijk} = \frac{1}{n+1}(h_{ij}I_k + h_{jk}I_i + h_{ki}I_j),
\]
so the Matsumoto torsion $\mathbf{M}_y$ vanishes.
\end{proof}

If a tangent vector $y$ satisfies $\mathbf{M}_y=0$, then we call $y$ a \emph{Matsumoto direction}.  The above lemma
simply says that every direction must belong to at least one of two cases:  polar, or Matsumoto. 

\begin{lemma}\label{lem:polar}
Suppose $(M,F)$ is a Finsler manifold of sectional flag curvature with $\dim M\geq 3$.  Fix a point $x$ in $M$.  
If the measure of polar directions in $T_xM$ is nonzero, then the flag curvature is isotropic at $x$;
if the measure of polar directions is zero, then the Minkowski norm $F|_x$ is of Randers type.
\end{lemma}
\begin{proof}
Let $U$ be the set of polar directions in $T_xM$.  If the measure of $U$ is zero, 
then by Lemma \ref{lem:key} the Matsumoto torsion $\mathbf{M}_y$ vanishes whenever $y$ does not belong to $U$, 
and thus $\mathbf{M}_y$ vanishes everywhere by continuity.  As a result, $F|_x$ is a Randers norm.

Now suppose the measure of $U$ is nonzero.   
If $y$ is a polar direction, $\mathbf{R}_y(v,v) = \kappa(y) F^2(y) h_y(v,v)$, then the flag curvature $K(y,v) = \kappa(y)$ only
depends on the flagpole $y$.  For any two linearly independent polar directions $y_1$, $y_2$ in $U$, we have
$\kappa(y_1) = K(y_1, y_2) = K(y_2,y_1) = \kappa(y_2)$.  This shows that the flag curvature is constant in $U$.  Denote
this constant by $c$. 

Let $y$, $z$ be nonzero vectors in $T_xM$ which are polar and non-polar directions, respectively.  
Then $K(z,y) = K(y,z) = c$.  It follows that $\mathbf{R}_z(y,y) = c\cdot F^2(z)h_z(y,y)$ holds for any $y$ in $U$.  Thus,
the zero set of the quadratic form $R_z - c F^2(z)h_z$ has positive measure.  Consequently,  this quadratic form 
is identically zero;  $R_z(v,v) = c\cdot F^2(z)h_z(v,v)$ holds for any $v$ in $T_xM$.  This proves that the flag curvature
of any flag with flagpole $z$ equals $c$.
\end{proof}

\section{Proof of the main theorems}
\label{sec:prf}

Using the above lemmas, we now prove the main result of Chen-Zhao \cite{ChenZhao2010}.

\begin{theorem}\label{thm:cz}
Let $M$ be a manifold of dimension $\geq 3$, and $F=\alpha+\beta$ be a Randers metric on $M$ with $\beta\neq0$.
If $F$ is of sectional flag curvature, then it has constant flag curvature.
\end{theorem}
\begin{proof}
The well-known theorem of Matsumoto \cite{Matsumoto} states that, in dimensions $\geq 3$,  
the metric $F$ is of Randers type if and only if the Matsumoto torsion vanishes, $\mathbf{M}_y=0$.
Equivalently
\[
	\mathbf{C}_y(u,v,w) = \frac{1}{n+1}\big(h_y(u,v)\mathbf{I}_y(w) +h_y(v,w)\mathbf{I}_y(u)+h_y(w,u)\mathbf{I}_y(v)\big).
\]
Thus we have
\begin{equation}\label{eq:t1}
	\mathbf{C}_y(v,v,\mathbf{R}_y(v)) = \frac{1}{n+1}\big(h_y(v,v)\mathbf{I}_y(\mathbf{R}_y(v)) + 2\mathbf{R}_y(v,v)\mathbf{I}_y(v)\big).
\end{equation}
Moreover,  the vanishing of Matsumoto torsion implies
\[
	\mathbf{C}_y(v,v,v)=\frac{3}{n+1}h_y(v,v)\mathbf{I}_y(v).
\]
Taking dynamical derivatives yields $\dot{\mathbf{C}}_y(v,v,v) = \frac{3}{n+1} h_y(v,v)\dot{\mathbf{I}}_y(v)$, and
\begin{equation}\label{eq:ddC}
	\ddot{\mathbf{C}}_y(v,v,v) = \frac{3}{n+1} h_y(v,v)\ddot{\mathbf{I}}_y(v).
\end{equation}
Substituting the above result into (\ref{eq:Rh=CR}) and rearranging,  we have
\[
	h_y(v,v)\big(\mathbf{I}_y(\mathbf{R}_y(v)) - 3\ddot{\mathbf{I}}_y(v)\big) = 4\mathbf{R}_y(v,v)\mathbf{I}_y(v).
\]
If $\beta\neq 0$, then $\mathbf{I}_y\neq 0$.  It follows from the above equation that $h_y(v,v) | \mathbf{R}_y(v,v)$,  
so $y$ is a polar direction.  As a result, we have proved that every direction is polar, thus the flag curvature is isotropic by 
Lemma \ref{lem:polar}.   Schur's lemma then ensures that the flag curvature is constant. 
\end{proof}

Based on the above theorem,  we finally arrive at the following conclusion.

\begin{theorem}
Let $(M,F)$ be a Finsler manifold of sectional flag curvature.  Then at any point $x$ in $M$, 
either the Minkowski norm $F|_x$ is Euclidean,  or the flag curvature is isotropic at $x$.
\end{theorem}
\begin{proof}
As we have shown, one only needs to consider the case $\dim M\geq 3$.  By Lemma \ref{lem:key}, every nonzero tangent vector $y$
is either a polar direction, or a Matsumoto direction.

We write $M=M_1\cup M_2$, where $M_1$ is the set of points $x$ such that the measure of polar directions in $T_xM$ is nonzero,
and $M_2$ is the set of points such that the measure of polar directions is zero.

If $x$ belongs to $M_1$,  then Lemma \ref{lem:polar} ensures that the flag curvature is isotropic at $x$.  So we only need to consider
the case $x\in M_2$.  If the measure of $M_2$ is zero, then by continuity, the flag curvature is isotropic everywhere.  If the measure 
of $M_2$ is nonzero, then $(M_2,F)$ is a Randers manifold.  By theorem \ref{thm:cz},  either the metric reduces to a Riemannian one,
or the flag curvature is constant.  Thus the theorem is proved.
\end{proof}

From the above proof,  it seems possible that a Finsler manifold of sectional flag curvature is glued by two manifolds,  of which
one is a Riemannian manifold and the other is a non-Riemannian manifold of isotropic flag curvature.   However, we could not
find a concrete example.

\end{document}